\documentclass [12pt, leqno,fleqn]{amsart}

\usepackage{amsmath,amstext,amsthm,amssymb,mathtools,amsxtra}
\usepackage[bookmarksopen,bookmarksdepth=3,colorlinks,citecolor=red,pagebackref,hypertexnames=true]{hyperref}
\usepackage[backrefs,msc-links,nobysame,initials]{amsrefs}

\allowdisplaybreaks

\usepackage{amsmath,amstext,amsthm,amssymb,amsxtra}
\usepackage{enumerate}
\usepackage{amsfonts}
\usepackage{amssymb}
\usepackage{amsmath}
\usepackage{graphicx}
\usepackage{graphicx,color}
\usepackage[normalem]{ulem}
\usepackage{todonotes}

\theoremstyle{plain}
\newtheorem{lem}{Lemma}
\newtheorem{cor}{Corollary}
\newtheorem{thm}{Theorem}
 \theoremstyle{definition}
\newtheorem*{remark}{Remark}

\begin{document}
\title[Solyanik estimates and the  continuity of halo functions]{Solyanik estimates and local H\"older continuity of halo functions of geometric maximal operators}
\author{Paul Hagelstein}
\address{Department of Mathematics, Baylor University, Waco, Texas 76798}
\email{\href{mailto:paul_hagelstein@baylor.edu}{paul\!\hspace{.018in}\_\,hagelstein@baylor.edu}}
\thanks{P. H. is partially supported by a grant from the Simons Foundation (\#208831 to Paul Hagelstein).}

\author{Ioannis Parissis}
\address{Department of Mathematics, Aalto University, P. O. Box 11100, FI-00076 Aalto, Finland}
\email{\href{mailto:ioannis.parissis@gmail.com}{ioannis.parissis@gmail.com}}
\thanks{I. P. is supported by the Academy of Finland, grant 138738.}

\begin{abstract}
Let $\mathcal{B}$ be a homothecy invariant basis consisting of convex sets in $\mathbb{R}^n$, and define the associated geometric maximal operator $M_{\mathcal{B}}$ by
\[
M_{\mathcal{B}} f(x) \coloneqq \sup_{x \in R \in \mathcal{B}}\frac{1}{|R|}\int_R |f|
\]
and the halo function $\phi_{\mathcal{B}}(\alpha)$ on $(1,\infty)$ by
\[
\phi_{\mathcal B}(\alpha) \coloneqq  \sup_{E \subset \mathbb{R}^n :\, 0 < |E| < \infty}\frac{1}{|E|}|\{x\in \mathbb{R}^n : M_{\mathcal{B}} \chi_E (x) >1/\alpha\}|.
\]
 It is shown that if $\phi_{\mathcal{B}}(\alpha)$ satisfies the Solyanik estimate $\phi_{\mathcal B}(\alpha) - 1 \leq C (1 - \frac{1}{\alpha})^p$ for $\alpha\in(1,\infty)$ sufficiently close to 1 then $\phi_{\mathcal{B}}$ lies in the H\"older class $ C^p(1,\infty)$.  As a consequence we obtain that the halo functions associated with the Hardy-Littlewood maximal operator and the strong maximal operator on $\mathbb{R}^n$ lie in the H\"older class $C^{1/n}(1,\infty)$.
\end{abstract}
\maketitle
\raggedbottom
%%%%%%%%%%%%%%%%%%%%%%%%%%%%%  Title

%%%%%%%%%%%%%%%%%%%%%%%%%%%%%% SECTION SECTION SECTION
 \section{Introduction}
From the time of the seminal paper \cite{hardylittlewood} of Hardy and Littlewood, geometric maximal functions have played a central role in analysis.   For example, the Hardy-Littlewood maximal operator $M_{\text{HL}}$ has been used in a proof of the Lebesgue differentiation theorem as well as in proofs of the $L^p$ boundedness of a wide class of singular integral operators; \cite{SI} provides a well-known exposition of these facts.

 A key property that the maximal operator $M_{\text{HL}}$ satisfies is the so-called \emph{weak type (1,1) estimate}:
\[
|\{x\in \mathbb{R}^n : M_{\text{HL}}f(x) > \alpha\}| \leq \frac{C_n}{\alpha}\int_{\mathbb{R}^n} |f|,\quad \alpha\in(0,\infty).
\]   It is this property that enables us to show that the collection of cubes or balls in $\mathbb{R}^n$ differentiates $L^1(\mathbb{R}^n)$.  Now, the \emph{strong maximal operator} $M_{\operatorname S}$, defined by taking maximal averages of a function over rectangular parallelepipeds whose sides are parallel to the coordinate axes, does not satisfy a weak type $(1,1)$ condition, although it does satisfy a weak type estimate of the form
\[
|\{x \in \mathbb{R}^n : M_{\operatorname S}f(x) > \alpha\}| \leq C_n \int_{\mathbb{R}^n} \frac{|f|}{\alpha}\left(1 + \log^+ \left( \frac{|f|}{\alpha}\right)\right)^{n-1},\quad \alpha\in(0,\infty).
\]
This weak type estimate can be used to show that the collection of rectangular parallelepipeds whose sides are parallel to the axes differentiates functions that are locally in $L(\log L)^{n-1}(\mathbb{R}^n)$; details in this regard may be found in \cite{Gu}.

In order to describe more general results along these lines we introduce some terminology. A \emph{basis} $\mathcal B$ is a collection of bounded open sets in $\mathbb R^n$. A collection $\mathcal B$ is called a \emph{density basis} if it differentiates functions in $L^\infty(\mathbb R^n)$. Given a basis $\mathcal{B}$ we explicitly define the maximal operator $M_\mathcal{B}$ by
\[
M_{\mathcal{B}}f(x)\coloneqq  \sup_{x \in R \in \mathcal{B}}\frac{1}{|R|}\int_R |f|
\]
if $x\in \cup_{B\in\mathcal B} B$ while we set $M _{\mathcal B} f (x)\coloneqq 0$ otherwise. We use the special notations $M_{\operatorname{HL},b}$ when $\mathcal B$ is the collection of all Euclidean balls in $\mathbb R^n$, $M_{\operatorname{HL},c}$ when $\mathcal B$ is the basis of all cubes in $\mathbb R^n$ with sides parallel to the coordinate axes, and $M_{\operatorname S}$ when $\mathcal B$ consists of all rectangular parallelepipeds in $\mathbb R^n$ with sides parallel to the coordinate axes. If a maximal operator $M_\mathcal{B}$ associated with  a basis $\mathcal{B}$ satisfies a weak type $(\Phi, \Phi)$ estimate, with $\Phi$ being a convex non-negative, non-decreasing function in $(0,\infty)$ with $\Phi(0)=0$, the basis $\mathcal{B}$ is known to differentiate functions for which $\Phi(f)$ is locally integrable. For this reason, given a maximal operator $M_{\mathcal{B}}$, it is highly desirable to place bounds on its distribution function $|\{x \in \mathbb{R}^n : M_{\mathcal{B}}f(x) > \alpha\}|$; that enables us to establish differentiation results for the basis $\mathcal{B}$.

A somewhat weaker estimate on a maximal operator is a so-called \emph{Tauberian condition}.  A maximal operator $M_{\mathcal{B}}$ associated with a basis $\mathcal{B}$ is said to satisfy a Tauberian condition with respect to $\alpha\in(0,1)$ if there exists a constant $C >0$ such that
\[
|\{x \in \mathbb{R}^n : M_\mathcal{B} \chi_E (x) > \alpha\}| \leq C |E|
\]
holds for all measurable sets $E \subset \mathbb R ^n $.     This condition is quite useful. C\'ordoba and Fefferman related Tauberian conditions of maximal operators to $L^p$ bounds of multiplier operators in \cite{CorF}, and Hagelstein and Stokolos showed in \cite{hs} that if $\mathcal B$ is a homothecy invariant basis consisting of convex sets and the maximal operator $M_\mathcal{B}$ satisfies a Tauberian condition with respect to some $  \alpha \in(0, 1)$ then $M_{\mathcal B}$ must be bounded on $L^p(\mathbb{R}^n)$ for sufficiently large $p$.  Subsequent papers extending these ideas include \cite{hlp} and \cite{hs2}.

The \emph{halo function} $\phi_{\mathcal{B}}$ associated with a density basis $\mathcal{B}$ is defined as
\[
\phi_{\mathcal{B}}(\alpha) \coloneqq  \sup_{E :\, 0 < |E| < \infty} \frac{1}{|E|}|\{x \in \mathbb{R}^n : M_\mathcal{B} \chi_E (x) > \frac{1}{\alpha}\}|
\]
for $\alpha\in (1,\infty)$, and by convention it is defined as $\phi_{\mathcal{B}}(\alpha) \coloneqq  \alpha$ for $\alpha\in[0,1]$. The growth of the halo function $\phi_\mathcal{B}(\alpha)$ as $\alpha \rightarrow \infty$  enables us to establish \emph{weak type bounds} on $M_{\mathcal B}$; in particular, if $\phi_{\mathcal B}(\alpha) \leq C \alpha^p$ for $\alpha > 1$, then $M_{\mathcal{B}}$ is of restricted weak type $(p,p)$ and accordingly $\mathcal{B}$ differentiates all functions which are locally in $L^q(\mathbb{R}^n)$ for $q > p$.   A prominent unsolved problem in differentiation theory is  the \emph{halo conjecture} which asserts that if $\mathcal{B}$ is a homothecy invariant density basis, then $\mathcal{B}$ must differentiate any measurable function $f$ for which $\phi_{\mathcal{B}}(f)$ is locally integrable. Partial results regarding the halo conjecture may be found in \cite{Gu}, \cite{Hayes} and \cite{soria} .

For these reasons, it is the issue of the growth properties of halo functions that has received the majority of attention in the field of differentiation theory in recent decades.  A fundamental but until recently overlooked issue is that of \emph{continuity and smoothness} of  halo functions. Beznosova and Hagelstein proved in \cite{bh} that the halo function of a  density basis  must be continuous on $[0,1]$ and $(1,\infty)$.   However, they also provided an example of a density basis consisting of nonconvex sets whose halo function exhibits a jump discontinuity at 1.

These results  immediately motivate a closer study of the behavior of halo functions near $1$.  The first results in this regard are due  to A. A. Solyanik, who proved in \cite{solyanik} that the halo functions associated with the centered Hardy-Littlewood maximal operator and the strong maximal operator tend to $1$ as $\alpha\to 1$.   Similar estimates were shown for the uncentered Hardy-Littlewood maximal operator, defined with respect to balls,  by Hagelstein and Parissis in \cite{hp13}, and analogues of these results were proved in the weighted case by Hagelstein and Parissis in \cite{hp14}.

Our previous work on Solyanik estimates was motivated primarily out of intrinsic interest.   However, Michael Lacey has subsequently brought to our attention that he and Sarah Ferguson implicitly used Solyanik-type estimates for the strong maximal operator in their work establishing a commutator estimate enabling one to give a characterization of the product BMO space $\operatorname{BMO}(\mathbb{R}_+^2 \times \mathbb R_{+}^2)$ of Chang and R. Fefferman in terms of commutators; see the appendix of their paper \cite{laceyferg2002}. Two other papers where Solyanik-type estimates were implicitly used  include \cite{cab} by Cabrelli, Lacey, Molter, and Pipher as well as \cite{laceyterwilleger} by Lacey and Terwilliger.

Given a density basis $\mathcal{B}$ it will be convenient for us to define the \emph{sharp Tauberian constant $C_\mathcal{B}(\alpha)$} by $\phi_{\mathcal{B}}(\frac{1}{\alpha})$ for $0 < \alpha < 1$.  In particular
\[
C_\mathcal{B}(\alpha) \coloneqq  \sup_{E : \, 0 < |E| < \infty} \frac{1}{|E|}|\{x \in \mathbb{R}^n : M_\mathcal{B} \chi_E (x) > \alpha \}|.
\]
Recall that if a density basis	 $\mathcal{B}$ is homothecy invariant then $C_\mathcal{B}(\alpha)$ is always finite for $\alpha \in(0,1)$; see \cite{busemannfeller1934}. A more precise quantitative version of the Solyanik estimates discussed above, collectively due to Hagelstein, Parissis, and Solyanik, is the following.

%%%%%%%%%%%%%%%%%%%%%%%%%%%%%% THEOREM THEOREM THEOREM
\begin{thm}\cites{hp13, solyanik}\label{t.t1}
% Let $M_{\operatorname{HL},b}$ and $M_{\operatorname{HL},c}$ denote the uncentered Hardy-Littlewood maximal operator with respect to balls, and cubes, respectively, and let $M_{\operatorname{S}}$ denote the strong maximal operator.
Let $C_{\operatorname{HL},b}$, $C_{\operatorname{HL},c}$, and $C_{\operatorname{S}}$ denote the sharp Tauberian constants for the uncentered Hardy-Littlewood maximal operator with respect to balls, the uncentered Hardy-Littlewood maximal operator with respect to cubes, and the strong maximal operator, respectively. Then we have the following asymptotic estimates for $\alpha\in(0,1)$ sufficiently close to $1$:
\[
 \begin{split}
C_{\operatorname{HL},b} (\alpha) - 1  &\lesssim_n  (1 - \alpha)^{\frac{1}{n+1}}
\\
C_{\operatorname{HL},c}(\alpha) - 1 &\sim_n (1 - \alpha)^{\frac{1}{n}} ,
\\
&\text{and}
\\
C_{\operatorname{S}}(\alpha) - 1 &\sim_n (1 - \alpha)^{\frac{1}{n}} .
 \end{split}
\]
\end{thm}
%%%%%%%%%%%%%%%%%%%%%%%%%%%%%% THEOREM THEOREM THEOREM
The estimates for $C_{\operatorname{HL},c}$ and $C_{\operatorname{S}}$ are sharp in the sense that the exponent $\frac{1}{n}$ cannot be replaced by any larger exponent.  Whether or not the exponent associated with $C_{\operatorname{HL},b}$ can be improved to $\frac{1}{n}$ or possibly $\frac{2}{n+1}$ is an open problem; see \cite{hp13}.

The purpose of this paper is to show that the above quantitative \emph{Solyanik estimates} may be used to establish H\"older continuity results for $C_{\operatorname{HL},b}$, $C_{\operatorname{HL},c}$, and $C_{\operatorname S}$, and accordingly yield H\"older continuity results for the corresponding halo functions.   We will moreover see that, given a homothecy invariant density basis $\mathcal{B}$ consisting of convex sets, a Solyanik estimate of the form
\[
C_{\mathcal{B}}(\alpha) - 1 \sim_n (1 - \alpha)^{p},\quad \alpha\to 1^-,
\]
implies that $\phi_{\mathcal{B}}$ lies in the H\"older class $C^p(0,1)$. A key ingredient of the proof of this result will be the careful use of the Calder\'on-Zygmund decomposition to show that \emph{halo sets} of the form
\[
\mathcal{H}_{\mathcal{B}, \alpha}(E) \coloneqq  \{x \in \mathbb{R}^n : M_{\mathcal{B}}\chi_E(x) > \alpha\}
\]
satisfy an imbedding relation
\[
\mathcal{H}_{\mathcal{B}, \alpha}(E) \subset \mathcal{H}_{\mathcal{B}, \alpha (1+ \gamma(\alpha)\delta) } (\mathcal{H}_{\mathcal{B}, 1 - c_n\delta}(E))
\]
whenever $\delta\lesssim_n 1-\alpha$. Here $\gamma(\alpha)\sim_n \min(\alpha,1-\alpha)^{2n}$ and $c_n>0$ is a dimensional constant.

%%%%%%%%%%%%%%%%%%%%%%%%%%%%%% SECTION SECTION SECTION
\section*{Notation} In this paper we will make frequent use of the following notational conventions. We write $C,c>0$ for numerical constants that can change value even in the same line of text. The presence of a subscript as in $C_\tau$ denotes dependence on some parameter $\tau$.  We write $A \lesssim B$ whenever $A \leq CB$ for some constant $C > 0$ and $A \sim B$ if $A \lesssim B$ and $B \lesssim A$.   We write $A \lesssim_\tau B$ whenever the implied constant depends on the parameter $\tau$, and $A \sim_\tau B$ if $A \lesssim_\tau B$ and $B \lesssim_\tau A$. Also, given an interval $I \subset \mathbb{R}$, we say that a function $f$ lies in the H\"older class $C^p(I)$ whenever for any compact set $K\subset I$ we have  $|f(x) - f(y)| \lesssim_K |x - y|^p$ for all $x,y \in K$.   This condition corresponds to $f$ being \emph{locally H\"older continuous} with exponent $p$ in $I$.   We are following here the notation and terminology found, for instance, in \cite{gilbargtrudinger}. We many times refer to a rectangular parallelepiped $R$ as a \emph{rectangle} in $\mathbb R^n$. Any rectangle $R$ gives rise to a dyadic grid consisting of homothetic copies of $R$ simply by bisecting each side of $R$ and iterating. Thus any $R$ gives rise to $2^n$ \emph{dyadic children} and any dyadic descendant $S$ of $R$ is contained in a unique \emph{dyadic parent} that we always denote by $S^{(1)}$.

% is contained in a unique \emph{dyadic parent} which we always denote by $R^{(1)}$.

%%%%%%%%%%%%%%%%%%%%%%%%%%%%%% SECTION SECTION SECTION
\section{Embedding of Halo Sets associated with bases of rectangular parallelepipeds}

In this section we provide the statement and proof of the following theorem regarding the embedding of halo sets associated with bases of rectangular parallelepipeds.

%%%%%%%%%%%%%%%%%%%%%%%%%%%%%% THEOREM THEOREM THEOREM
\begin{thm}\label{t.t2} Let $\mathcal{B}$ be a homothecy invariant collection consisting of rectangular parallelepipeds in $\mathbb{R}^n$.  Given a measurable set $E \subset \mathbb{R}^n$ of finite measure  and $\alpha \in(0,1)$ we define the associated halo set $\mathcal{H}_{\mathcal{B},\alpha }(E)$ by
\[
\mathcal{H}_{\mathcal{B}, \alpha}(E) \coloneqq  \left\{x \in \mathbb{R}^n : M_{\mathcal{B}}\chi_E(x) > \alpha \right\}.
\]
 Then for all $\xi,\delta\in(0,1)$ with $\alpha<1-\delta<\xi$ we have
\[
\mathcal{H}_{\mathcal{B},\alpha}(E) \subset \mathcal{H}_{\mathcal{B},\alpha(1 + \frac{1-\xi}{ 2^n})}(\mathcal{H}_{\mathcal{B},1-\delta}(E)).
\]
\end{thm}
%%%%%%%%%%%%%%%%%%%%%%%%%%%%%% THEOREM THEOREM THEOREM

%%%%%%%%%%%%%%%%%%%%%%%%%%%%%% PROOF PROOF PROOF
\begin{proof}
Let $E \subset \mathbb{R}^n$ be a set of finite measure and fix $\alpha,\delta,\xi \in(0,1)$ with $\alpha<1-\delta<\xi$.  Let $x \in \mathcal{H}_{\mathcal{B},\alpha}(E)$.  Then there exists a rectangle $R \in \mathcal{B}$ containing $x$ such that
\[
\frac{1}{|R|}\int_R \chi_E > \alpha.
\]
Using the homothecy invariance of $\mathcal{B}$, we may assume without loss of generality that the rectangle $R$ is sufficiently large so that $|E\cap R|/|R|\leq \xi <1$ holds as well.

We will now consider a modified Calder\'on-Zygmund decomposition of $\chi_{E\cap R}$ with respect to the dyadic grid generated by $R$ at level $\xi$.
  Along the same lines as the standard Calder\'on-Zygmund decomposition with respect to cubes we see that there exists a collection $\{R_j\}_j$ of dyadic  subrectangles of $R$, each $R_j$ strictly contained in $R$, satisfying
\begin{align*}
&\tag{i}  E\cap R \subset \bigcup_j R_j \quad\text{a.e.},
\\
& \tag{ii} \frac{1}{|R_j|}\int_{R_j}\chi_E > \xi\;, \quad \text{and}
\\
&\tag{iii} \textup{ if $S$ is any dyadic ancestor of $R_j$ contained in $R$, then } \frac{1}{|S|}\int_S \chi_E \leq \xi.
\end{align*}
We recall that for a dyadic rectangle $S$ we denote by $S^{(1)}$ its unique dyadic parent and note that for all $j$ we have $R_j ^{(1)}\subseteq R$. Let now $\{R_{j_k} ^{(1)}\}_k$ be the maximal dyadic rectangles in the collection $\{R_j ^{(1)}\}_j$. Clearly
\[
 \big|\bigcup_k R_{j_k}\big| \geq \frac{1}{2^n}\big|\bigcup_j R_ j \big|.
\]
By continuity of the Lebesgue measure, for each $k$ there exists a homothetic copy $S_k$ of $R$ such that
\[
R_{j_k}\subsetneq S_k \subseteq R_{j_k} ^{(1)}\quad\text{and}\quad \frac{|S_k\cap E|}{|S_k|} = \xi.
\]
Since for each $k$ we have $S_k\subseteq R_{j_k} ^{(1)}$ and the $R_{j_k}^(1)$'s are maximal we have that the $S_k$'s are pairwise disjoint. Furthermore note that $S_k \subset \mathcal H_{\mathcal B,1-\delta}(E)$ since $\xi>1-\delta$. We conclude that
\[
\begin{split}
|(\mathcal{H}_{\mathcal{B}, 1 - \delta}(E) \backslash E) \cap R | &\geq \sum_k \big|(\mathcal{H}_{\mathcal{B}, 1 - \delta}(E) \setminus E) \cap {S}_k\big|\geq \sum_k |S_k\setminus E|
\\
& =(1-\xi) \sum_k|S_k| \geq  (1-\xi)  \sum_k |R_{j_k}|\geq \frac{1-\xi}{2^n} \big| \bigcup_j R_j \big|
\\
& \geq \frac{1-\xi}{2^n} |E\cap R|>\frac{1-\xi}{2^n} \alpha |R|.
\end{split}
\]
The previous calculation immediately implies
\[
\left|\mathcal{H}_{\mathcal{B}, 1 - \delta}(E) \cap R\right| >   \alpha\Big(1 + \frac{1-\xi}{  2^n}\Big)|R|
\]
so that
\[
R \subset \mathcal{H}_{\mathcal{B}, \alpha(1 + \frac{1-\xi}{  2^n})}(\mathcal{H}_{\mathcal{B}, 1 - \delta}(E)).
\]
Hence we can conclude that
\[
\mathcal{H}_{\mathcal{B}, \alpha}(E) \subset \mathcal{H}_{\mathcal{B}, \alpha(1 + \frac{1-\xi}{  2^n})}(\mathcal{H}_{\mathcal{B}, 1 - \delta}(E))
\]
as we wanted.
\end{proof}
%%%%%%%%%%%%%%%%%%%%%%%%%%%%%% PROOF PROOF PROOF

%%%%%%%%%%%%%%%%%%%%%%%%%%%%%% SECTION SECTION SECTION
\section{Local H\"older continuity of the  Halo Function of the Strong Maximal Operator}
As $|\mathcal{H}_{\mathcal{B}, \alpha}(E)| \leq C_{\mathcal{B}}(\alpha)|E|$, the following corollary follows by Theorem \ref{t.t2} and the results in \cite{bh} regarding continuity of Tauberian constants by letting $\xi\to (1-\delta)^+$.

%%%%%%%%%%%%%%%%%%%%%%%%%%%%%% COROLLARY COROLLARY COROLLARY
\begin{cor}\label{cor1}
Let $\mathcal{B}$ be a homothecy invariant density basis consisting of rectangular parallelepipeds in $\mathbb{R}^n$, $\alpha\in(0,1)$, and let $C_{\mathcal{B}}(\alpha)$ be the associated sharp Tauberian constant of $\mathcal{B}$ with respect to $\alpha$.  Then for all $\delta\in(0,1-\alpha)$ we have
\[
C_{\mathcal{B}}(\alpha) \leq C_{\mathcal{B}}\big(\alpha(1 + \frac{\delta}{ 2^n})\big)C_\mathcal{B}(1 - \delta).
\]
\end{cor}
%%%%%%%%%%%%%%%%%%%%%%%%%%%%%% COROLLARY COROLLARY COROLLARY

We shall now see that this corollary enables us to prove that the sharp Tauberian constant $C_{\operatorname S}(\alpha)$, associated with the strong maximal operator $M_{\operatorname S}$, satisfies a local H\"older continuity condition, with in fact $C_{S}  \in C^{\frac{1}{n}}(0,1)$.  We shall need the following technical lemma.

%%%%%%%%%%%%%%%%%%%%%%%%%%%%%% LEMMA LEMMA LEMMA
\begin{lem}\label{l.l1}
Let $\mathcal{B}$ be a homothecy invariant density basis consisting of rectangular parallelepipeds in $\mathbb{R}^n$.  Suppose that there exists $\alpha_o\in(0,1)$ such that the inequality $C_{\mathcal{B}}(\alpha) - 1 \lesssim_{ \mathcal B}(1 - \alpha)^p$ holds for all $\alpha\in(\alpha_o,1)$ and for some fixed $ p\in(0,1)$. Then $C_{\mathcal{B}}(\alpha)$ is locally H\"older continuous with exponent $p$ on $(0,1)$, that is, $C_{\mathcal{B}} \in C^p(0,1)$.
\end{lem}
%%%%%%%%%%%%%%%%%%%%%%%%%%%%%% LEMMA LEMMA LEMMA

%%%%%%%%%%%%%%%%%%%%%%%%%%%%%% PROOF PROOF PROOF
\begin{proof}  Let us fix a compact set $K\in (0,1)$ and let $m_K,M_K\in(0,1)$ be such that $m_K\leq x\leq M_K$ for all $x\in K$. By the results in \cite{bh}  we have that $C_{\mathcal B}$ is continuous in $(0,1)$ thus $\sup_{x\in K} C_{\mathcal B}(x)\lesssim_{ \mathcal B,K} 1$.

We first consider the case  $x,y\in K$ with  $0<y-x<\min\big(\frac{1-M_K}{2^n} m_K, \frac{1-\alpha_o}{2^n} m_K\big)\eqqcolon \eta$. We write
\[
C_{\mathcal B}(x)-C_{\mathcal B}(y)= C_{\mathcal B}(x)-C_{\mathcal B}\big(x(1+2^n\frac{y-x}{2^nx})\big).
\]
Since
\[
 2^n\frac{y-x}{x} < 2^n\frac{1-M_K}{2^n m_K} m_K= 1-M_K\leq 1-x,
\]
we get by Corollary~\ref{cor1} that
\[
 C_{\mathcal B}(x)-C_{\mathcal B}(y)\leq C_{\mathcal B}\big(y) \Big[C_{\mathcal B}\big(1-2^n\frac{y-x}{x}\big)-1\Big]\lesssim_{ \mathcal B,K }  \Big[C_{\mathcal B}\big(1-2^n\frac{y-x}{x}\big)-1\Big].
\]
Since
\[
 1-2^n\frac{y-x}{x} >1- 2^n\frac{1-\alpha_o}{2^n m_K} m_K=\alpha_o
\]
we get by the hypothesis of the lemma that
\[
  C_{\mathcal B}(x)-C_{\mathcal B}(y) \lesssim_{\mathcal{B},K} \frac{(y-x)^p}{x^p}\lesssim_K (y-x)^p,\quad x,y\in K,\quad 0<y-x<\eta.
\]
We can now conclude that
\[
\sup_{{\substack {x,y\in K \\  0<y-x<\eta }}} \frac{|C_{\mathcal{B}}(x) - C_{\mathcal{B}}(y)|}{|x-y|^p}  \lesssim_{ \mathcal B,K } 1.
\]
On the other hand, if $x,y\in K$ with $y-x\geq\eta$ then the H\"older bound follows trivially since $\sup_{x,y\in K}|C_{\mathcal B}(x)-C_{\mathcal B}(y)|\lesssim_{\mathcal B,K}1$ so we are done.
\end{proof}
%%%%%%%%%%%%%%%%%%%%%%%%%%%%%% PROOF PROOF PROOF
As we have that the strong maximal operator satisfies the Solyanik estimate
\[
C_{\operatorname S}(\alpha) - 1 \sim_n (1 - \alpha)^\frac{1}{n},\quad \alpha\to 1^-,
\]
we immediately conclude the following.
%%%%%%%%%%%%%%%%%%%%%%%%%%%%%% COROLLARY COROLLARY COROLLARY
\begin{cor}
Let $C_{ \operatorname S}(\alpha)$ denote the sharp Tauberian constant of the strong maximal operator in $\mathbb{R}^n$ with respect to $\alpha\in(0,1)$.  Then
\[
C_{ \operatorname S} \in C^{1/n}( 0,1 ).
\]
\end{cor}
%%%%%%%%%%%%%%%%%%%%%%%%%%%%%% COROLLARY COROLLARY COROLLARY
Recall that, following our previous convention, the halo function is given as $\phi_{\operatorname S} (x) = C_{\operatorname S}(1/x)$ for $x > 1$ and $\phi_{\operatorname S} (x) = x$ for $x \leq 1$.  Hence the fact that $C_{\operatorname S}(\alpha) - 1 \lesssim_n (1 - \alpha)^{1/n}$  enables us to immediately establish the following continuity estimate for the halo function on all of $[0,\infty)$.
%%%%%%%%%%%%%%%%%%%%%%%%%%%%%% COROLLARY COROLLARY COROLLARY
\begin{cor}
Let $\phi_{\operatorname S} (\alpha)$ be the halo function associated with the strong maximal operator on $\mathbb{R}^n$.  Then
\[
\phi_{\operatorname S}	 \in C^{1/n}([0,\infty)) .
\]
\end{cor}
%%%%%%%%%%%%%%%%%%%%%%%%%%%%%% COROLLARY COROLLARY COROLLARY
The reasoning in this section applies not only to the strong maximal operator $M_{\operatorname S}$, but also the Hardy-Littlewood maximal operator $M_{\operatorname{HL},c}$ with respect to cubes.   Observing that the exponents in their Solyanik estimates are the same, see Theorem~\ref{t.t1}, we have the following.
%%%%%%%%%%%%%%%%%%%%%%%%%%%%%% COROLLARY COROLLARY COROLLARY
\begin{cor}\label{c4}
Let $\phi_{\operatorname{HL},c} (\alpha)$ be the halo function associated with the Hardy-Littlewood maximal operator with respect to cubes on $\mathbb{R}^n$.  Then
\[
\phi_{\operatorname{HL},c} \in C^{1/n}([0,\infty)) .
\]
\end{cor}
%%%%%%%%%%%%%%%%%%%%%%%%%%%%%% COROLLARY COROLLARY COROLLARY

%%%%%%%%%%%%%%%%%%%%%%%%%%%%%% SECTION SECTION SECTION
\section{Local H\"older continuity of the  Halo Function of the Hardy-Littlewood Maximal Operator averaging over balls.}

A basis consisting of rectangular parallelepipeds enjoys the advantage of enabling arguments along the lines of the Calder\'on-Zygmund decomposition employed in the proof of Theorem \ref{t.t2}.  To yield H\"older continuity estimates for the halo function of the Hardy-Littlewood maximal operator with respect to \emph{balls}, however, additional arguments must be made that we provide here.

%%%%%%%%%%%%%%%%%%%%%%%%%%%%%% THEOREM THEOREM THEOREM
\begin{thm}\label{t.balls}
Let $\phi_{\operatorname{HL}, b}(\alpha)$ be the halo function associated with the Hardy-Littlewood maximal operator with respect to balls on $\mathbb{R}^n$.  Then
\[
\phi_{\operatorname{HL},b}  \in C^{1/n+1}([0,\infty)).
\]
Moreover, the associated Tauberian constants $C_{\operatorname{HL},b}(\alpha)$ satisfy
\[
C_{\operatorname{HL},b}  \in C^{1/n}(0,1)
\]
and consequently
\[
\phi_{\operatorname{HL},b}  \in C^{1/n}(1,\infty)	.
\]
\end{thm}
%%%%%%%%%%%%%%%%%%%%%%%%%%%%%% THEOREM THEOREM THEOREM

%%%%%%%%%%%%%%%%%%%%%%%%%%%%%% PROOF PROOF PROOF
\begin{proof}
Let $E \subset \mathbb{R}^n$ be a set of finite measure and $\alpha\in(0,1)$. Let $x\in \mathcal H_{\operatorname{HL},b,\alpha}(E)$. By the definition of $ \mathcal H_{\operatorname{HL},b,\alpha}(E)$ there exists some ball $\tilde B\subset \mathbb R^n$ with $\tilde B\ni x$ and
\[
 \frac{1}{|\tilde B|}\int_{\tilde B} \chi_E >\alpha.
\]
Then there exists a ball $B\supset \tilde B$ such that
\[
 \frac{1}{|B|}\int_B \chi_E =\alpha.
\]
By scaling, we can also assume that $B$ is the unit ball of $\mathbb R^n$.  Let also $0<\epsilon<\min(\alpha,1-\alpha)$ be a small parameter to be chosen later.  We now denote by $\mathcal{C}_m$ the collection of all dyadic cubes of sidelength $2^{-m}$ which are strictly contained in $B$. We choose $m$ to be a large positive integer so that
\[
\Big|B \setminus \bigcup_{C \in \mathcal{C}_m}C\Big| <\epsilon|B|.
\]
Observing that the measure of the union of all dyadic cubes in $\mathcal{C}_m$ intersecting the boundary of $B$ is $\sim_n 2^{-m}$ and using elementary geometric arguments we see that we may assume $2^{-m} \sim_n \epsilon$.

Next we claim that there exists a cube $R\subseteq B$, with $|R|= 2^{-mn} \sim_n \epsilon^n$, such that
\begin{equation}\label{e.goodcube}
 \frac{\alpha-\epsilon}{1-\epsilon}\leq \frac{1}{|R|}\int_R \chi_E \leq \frac{\alpha}{1-\epsilon}.
\end{equation}
Indeed, if for all $C \in \mathcal{C}_m$ the left hand side inequality of \eqref{e.goodcube} failed, we would have
\[
\begin{split}
 |B \cap E| &= |\bigcup_{C\in\mathcal{C}_m} C \cap E| + | \big(B \setminus\bigcup_{C\in\mathcal{C}_m} C \big) \cap E)|
\\
&<\frac{\alpha-\epsilon}{1-\epsilon} \sum_{C\in\mathcal{C}_m} |C| +  |B|-\big|\bigcup_{C\in\mathcal{C}_m} C\big|=|B|-\frac{1-\alpha}{1-\epsilon}\big|\bigcup_{C\in\mathcal{C}_m} C\big|
\\
& \leq |B|-\frac{1-\alpha}{1-\epsilon} (1-\epsilon)|B|=\alpha|B|
\end{split}
\]
contradicting the choice of $B$. On the other hand, if for all $C \in \mathcal{C}_m$ the right hand side inequality of \eqref{e.goodcube} failed we would get
\[
\begin{split}
 |B\cap E| & \geq \sum_{C \in \mathcal{C}_m} |C\cap E|> \frac{\alpha}{1-\epsilon} \sum_{C \in \mathcal{C}_m}|C| =\frac{\alpha}{1-\epsilon} \big|\bigcup_{C\in\mathcal{C}_m} C\big|
\\
& \geq \frac{\alpha}{1-\epsilon} (1-\epsilon)|B|=\alpha|B|
\end{split}
\]
which also contradicts our hypotheses on $B$. Thus there exist dyadic cubes $R_1,R_2\in  \mathcal{C}_m$ such that
\[
 \frac{1}{|R_1|}\int_{R_1} \chi_E \geq\frac{\alpha-\epsilon}{1-\epsilon}\quad\text{and}\quad\frac{1}{|R_2|}\int_{R_2} \chi_E \leq \frac{\alpha}{1-\epsilon}.
\]
As $R_1$ can be mapped onto $R_2$, inside $B$, by a continuous rigid motion and using the intermediate value theorem we conclude that there exists a cube $R\subseteq B$ with $|R|=|R_1|=|R_2| = 2^{-mn}$ such that\eqref{e.goodcube} holds.

For $\delta\in(0,1-\frac{\alpha}{1-\epsilon}]$ we now perform a Calder\'on-Zygmund decomposition of $\chi_{E\cap R}$ with respect to  the dyadic grid generated by $R$, at level $1-\delta$, to get the Calder\'on-Zygmund cubes $\{R_j\}_j$, with $E\cap R \subseteq \cup_j R_j$ a.e., $|R_j\cap E|/|R_j|>1-\delta$, and the cubes $R_j$ being maximal with respect to this property. For every $j$ we now have $R_j ^{(1)} \subseteq R.$ Indeed, if not, then $R$ would be itself a Calder\'on-Zygmund cube. However we have $|R\cap E|/|R|\leq \alpha/(1-\epsilon)\leq1-\delta$.

Let $\{R_{j_k} ^{(1)}\}_k$ be maximal among the parents $R_j ^{(1)}$. We have
\[
 \big|\bigcup_k R_{j_k}\big| \geq \frac{1}{2^n}\big|\bigcup_j R_ j \big|.
\]
By the continuity of the Lebesgue measure we can find, for each $k$, a cube $S_k$ with
\[
 R_{j_k}\subsetneq S_k \subseteq R_{j_k} ^{(1)}\quad\text{and}\quad \frac{|S_k\cap E|}{|S_k|}=1-\delta.
\]
Note that the $S_k$'s are a.e. pairwise disjoint and for each $k$ we have $S_k\subseteq \mathcal H_{\operatorname{S},1-\eta}(E)$ for every $\eta\in(\delta,1)$. Thus for such $\eta$ we get
\[
\begin{split}
 |\left(\mathcal{H}_{\operatorname{S}, 1 - \eta}(E) \setminus E\right) \cap R| &\geq \sum_k |\left(\mathcal{H}_{\operatorname{S}, 1 - \eta}(E) \setminus E\right) \cap S_k|  \geq \frac{\delta}{  2^n}|E \cap R|.
\end{split}
\]
Accordingly
\[
\begin{split}
|(\mathcal{H}_{\operatorname{S}, 1 -  \eta}(E) \setminus E) \cap B|  &\geq\frac{\delta}{  2^n}|E \cap R|  \geq \frac{\delta}{  2^n}\frac{\alpha - \epsilon}{1 - \epsilon}|R| \gtrsim_n\delta\epsilon^n \frac{\alpha - \epsilon}{1 - \epsilon}|B|.
\end{split}
\]
Hence
\[
B \subseteq \mathcal{H}_{\operatorname{HL}, b, \alpha + C_n  \epsilon^n  \delta \frac{\alpha - \epsilon}{1 - \epsilon}}(\mathcal{H}_{\operatorname{S}, 1 - \eta}(E))
\]
for some dimensional constant $C_n>0$. We conclude that
\[
C_{\operatorname{HL},b}(\alpha) \leq C_{ \operatorname{HL},b}\Big(\alpha + C_n  \epsilon^n \delta \frac{\alpha - \epsilon}{1 - \epsilon}\Big)C_{\operatorname S}(1 -\eta).
\]
Finally, using the continuity results from \cite{bh} and letting $\eta\to\delta^+$ we get that for all $\delta\leq 1-\frac{\alpha}{1-\epsilon}$ we have
\[
C_{\operatorname{HL},b}(\alpha) \leq C_{ \operatorname{HL},b}\Big(\alpha + C_n  \epsilon^n \delta \frac{\alpha - \epsilon}{1 - \epsilon}\Big)C_{\operatorname S}(1 -  \delta).
\]
At this point we set $\epsilon\coloneqq \frac{1}{2}\min(\alpha,1-\alpha)$. Then for all $\alpha\in(0,1)$ we have $\frac{\alpha-\epsilon}{1-\epsilon}\gtrsim \alpha $ and
\[
1-\frac{\alpha}{1 -\epsilon}=\begin{cases} \frac{2-3\alpha}{2-\alpha},\quad &\alpha\leq \frac{1}{2} ,
\\
\frac{1-\alpha}{1+\alpha},\quad &\alpha>\frac{1}{2} .\end{cases}
\]
From this we readily see that our estimates are valid for all $\delta\lesssim 1-\alpha$. Using the fact that $C_{\operatorname{HL},b}(\alpha)$ is non-increasing in $\alpha$ we can summarize our result to the estimate
\begin{align}\label{e.hlb}
 C_{\operatorname{HL},b}(\alpha) \leq C_{ \operatorname{HL},b}\Big(\alpha (1+c_n\min(\alpha,1-\alpha)^{n} \delta)\big)C_{\operatorname S}(1 -  \delta),\quad \delta\lesssim  1-\alpha ,	
\end{align}
for some dimensional constant $c_n>0$.
The proof of the statement $C_{\operatorname{HL},b}\in C^\frac{1}{n}(0,1)$ is completed by using the Solyanik estimate $C_{\operatorname S}(\alpha)-1\sim_n (1-\alpha)^\frac{1}{n}$ as $\alpha\to 1^-$ as in the proof of Lemma~\ref{l.l1}.

By the definition of the halo function the argument above implies that $\phi_{\operatorname{HL}, b} \in C^{1/n}(1,\infty)$.  However, to yield H\"older continuity of $\phi_{\operatorname{HL},b}$ at $1$ we must use the Solyanik estimates for $C_{\operatorname{HL},b}$ itself , not the Solyanik estimates for $C_{\operatorname S}$ that we were able to use above in order to find smoothness estimates on $C_{\operatorname {HL},b}$ on compact subsets of $(0,1)$.  By Theorem \ref{t.t1} we have that $C_{\operatorname{HL},b}(\alpha) - 1 \lesssim_n (1-\alpha)^{1/ n+1 }$ as $\alpha\to 1^-$. This result combined with the Solyanik estimate for $C_{\operatorname S}$ and  estimate \eqref{e.hlb} above gives
\[
\phi_{\operatorname{HL}, b} \in C^{1/ n+1 } ([0,\infty)) ,
\]
as we wanted.
\end{proof}
%%%%%%%%%%%%%%%%%%%%%%%%%%%%%% PROOF PROOF PROOF

%%%%%%%%%%%%%%%%%%%%%%%%%%%%%% SECTION  SECTION SECTION
\section{Bases of Convex Sets}\label{s.convex}
We now show that, if $\mathcal{B}$ is a homothecy invariant density basis consisting of convex sets which satisfies a Solyanik estimate of the form $C_{\mathcal{B}}(\alpha) - 1 \lesssim_{\mathcal B}(1 - \alpha)^p$ for $\alpha$ sufficiently close to 1, then $\phi_{\mathcal{B}} \in C^{p} ([0,\infty)) $.  We need some preliminary results. A schematic for the lemmas that follow and their corresponding proofs is contained in Figure~\ref{f.schematic}.

The following technical lemma uses a classical lemma of Fritz John \cite{fj} and will help us reduce the study of regularity estimates for the sharp Tauberian constants associated with bases of convex sets to estimates concerning rectangles in $\mathbb R^n$. We will refer to it as the Fritz John lemma.

%%%%%%%%%%%%%%%%%%%%%%%%%%%%%% COROLLARY COROLLARY COROLLARY
\begin{lem}
Let $\Lambda\subseteq \mathbb R^n$ be a bounded convex set in $\mathbb{R}^n$. Then there exists a rectangular parallelepiped $R_\Lambda\subseteq\mathbb R^n$ such that
\[
 R_{\Lambda} \subseteq \Lambda \subseteq n^\frac{3}{2} R_{\Lambda}.
\]
\end{lem}
%%%%%%%%%%%%%%%%%%%%%%%%%%%%%% COROLLARY COROLLARY COROLLARY

%%%%%%%%%%%%%%%%%%%%%%%%%%%%%% PROOF PROOF PROOF
\begin{proof}
Let $\Lambda$ be a bounded convex set in $\mathbb{R}^n$.  As was proven by Fritz John in \cite{fj}, $\Lambda$ must contain an ellipsoid $\mathcal{E}_{\Lambda}$ such that
\[
\mathcal{E}_{\Lambda} \subseteq \Lambda \subseteq n \mathcal{E}_{\Lambda}.
 \]
Here the dilation $n\mathcal{E}_\Lambda$ is taken with respect to the center of the ellipsoid. Let $S_{\Lambda}$ be a rectangular parallelepiped of minimal volume containing $\mathcal{E}_{\Lambda}.$  By elementary geometry, we have $\mathcal{E}_{\Lambda} \subseteq S_{\Lambda} \subseteq n^{1/2}\mathcal{E}_{\Lambda}$ and hence
\[
n^{-\frac{1}{2}}S_{\Lambda} \subseteq 	\mathcal{E}_{\Lambda} \subseteq \Lambda \subseteq n\mathcal E_\Lambda \subseteq n S_\Lambda.
\]
The desired rectangle is now given as $R_\Lambda \coloneqq n^{-\frac{1}{2}}S_\Lambda$.
\end{proof}
%%%%%%%%%%%%%%%%%%%%%%%%%%%%%% PROOF PROOF PROOF
We proceed with a simple geometric lemma that quantifies the measure theoretic approximation of a convex set by finite unions of dyadic cubes.
%%%%%%%%%%%%%%%%%%%%%%%%%%%%%% LEMMA LEMMA LEMMA
\begin{lem}\label{l.l2}
Let $\Lambda\subseteq \mathbb R^n$ be a convex set such that, letting $Q \coloneqq [0,1]^n$ and $cQ$ the $c$-fold concentric dilate of $Q$, we have $Q \subset \Lambda \subset n^{3/2}Q$.  For every $\epsilon \in(0,1)$ there exists a collection of a.e. disjoint dyadic cubes $\{C_j\}_j$ contained in $\Lambda$, each of sidelength  $\sim_n \epsilon$, so that $|\Lambda \setminus \cup_j C_j| < \epsilon|\Lambda|$.
\end{lem}
%%%%%%%%%%%%%%%%%%%%%%%%%%%%%% LEMMA LEMMA LEMMA

%%%%%%%%%%%%%%%%%%%%%%%%%%%%%% PROOF PROOF PROOF
\begin{proof}
Let $\mathcal{C}_m$ denote the collection of dyadic cubes of sidelength $2^{-m}$, where $m$ is a non-negative integer, which are contained in $\Lambda$.  Suppose $x \in S$ and $\operatorname{dist}(x, \partial \Lambda) > \sqrt{n}2^{-m}$.   Then there exists a dyadic cube $C\in \mathcal C_m$ with $C\ni x$.  Using the convexity of $\Lambda$ we conclude that
\[
 \Big|\Lambda\setminus \bigcup_{C\in\mathcal C_m} \Big|\leq |\{x\in\Lambda:\, \operatorname{dist}(x,\partial \Lambda)\leq \sqrt{n} 2^{-m}\}|\lesssim_n 2^{-m}\sim_n 2^{-m}|\Lambda|.
\]
Choosing $2^{-m}\sim_n \epsilon$ proves the lemma.
\end{proof}
%%%%%%%%%%%%%%%%%%%%%%%%%%%%%% PROOF PROOF PROOF
We proceed by showing that if the average of $\chi_E$ with respect to some convex set $\Lambda$ is equal to $\alpha$ then $\chi_E$ must have average close to $\alpha$ with respect to some cube $R\subseteq \Lambda$, where we also have a control on the measure of the cube $R$.
%%%%%%%%%%%%%%%%%%%%%%%%%%%%%% LEMMA LEMMA LEMMA
\begin{lem}\label{l.l3}
Let $\alpha \in(0,1)$ and $0<\epsilon<\min(\alpha,1-\alpha)$ and let $\Lambda\subset \mathbb R^n$ be a convex set satisfying $Q \subset \Lambda \subset n^{3/2}Q$. Suppose that $E\subset \mathbb R^n$ is a measurable set of finite measure for which $\frac{1}{|\Lambda|}\int_\Lambda \chi_E= \alpha$. Then there exist a cube $ R \subseteq \Lambda$ with $|R|\sim_n \epsilon^n$ such that
\[
\frac{\alpha - \epsilon}{1 -\epsilon} \leq \frac{1}{|R|}\int_{R}\chi_E \leq \frac{\alpha}{1 -\epsilon}.
\]
\end{lem}
%%%%%%%%%%%%%%%%%%%%%%%%%%%%%% LEMMA LEMMA LEMMA

%%%%%%%%%%%%%%%%%%%%%%%%%%%%%% PROOF PROOF PROOF
\begin{proof} Let $\{C_j\}_j$ be the collection of dyadic cubes provided by Lemma~\ref{l.l2}. We have that $C_j\subseteq \Lambda$ for every $j$, $|C_j|\sim_n \epsilon^n$ and $|\Lambda\setminus \cup_j C_j|\leq \epsilon |\Lambda|$. Arguing as in the proof of Theorem~\ref{t.balls} we see that there exist dyadic cubes $R_1,R_2 \in\{C_j\}_j$ such that
\[
 \frac{1}{|R_1|}\int_{R_1}\chi_E \geq \frac{\alpha-\epsilon}{1 -\epsilon}\quad\text{and}\quad \frac{1}{|R_2|}\int_{R_2}\chi_E \leq \frac{\alpha}{1 -\epsilon}.
\]
As the convex hull of the union of any two $C_j$'s lies in $\Lambda$, by the intermediate value theorem we see there exists a cube $R \subset \Lambda$ of measure $\sim_n\epsilon^n$ that satisfies the conclusion of the lemma.
\end{proof}
%%%%%%%%%%%%%%%%%%%%%%%%%%%%%% PROOF PROOF PROOF
We will need to work with a smaller cube inside the one provided by the previous lemma. The following lemma summarizes the technical details of this construction.
%%%%%%%%%%%%%%%%%%%%%%%%%%%%%% LEMMA LEMMA LEMMA
\begin{lem}\label{l.rin} Let $\alpha\in(0,1)$, $0<\epsilon<\min(\alpha,1-\alpha)$, and $R$ be the cube provided by Lemma~\ref{l.l3}. There exists a cube $R^{\textup{in}}\subsetneq R$ with $|R|\sim_n |R^\textup{in}|$ and
	\[
	\frac{1}{2}\frac{\alpha-\epsilon}{1-\epsilon}	\leq \frac{1}{|R^\textup{in}|}\int_{R^\textup{in}}\chi_E \leq \frac{1}{2}\Big(1+\frac{\alpha}{1-\epsilon}\Big)\;.
	\]
\end{lem}
%%%%%%%%%%%%%%%%%%%%%%%%%%%%%% LEMMA LEMMA LEMMA

%%%%%%%%%%%%%%%%%%%%%%%%%%%%%% PROOF PROOF PROOF
\begin{proof} For $t\in (0,1)$ we have
	\[
	1-t^{-n}\Big[1-\frac{\alpha-\epsilon}{1-\epsilon}\Big] \leq \frac{1}{|tR |}\int_{tR}\chi_E \leq t^{-n}\frac{\alpha}{1-\epsilon}\;.
	\]
We now choose $t_o$ by setting
\[
t_o^{-n}\coloneqq  \min \Big( \big(1+\frac{1}{2}(\frac{1-\epsilon}{\alpha}-1)\big),\frac{1-\epsilon}{1-\alpha}-\frac{1}{2}\frac{\alpha-\epsilon}{1-\alpha}\Big)
\]
and define $R^\text{in}\coloneqq t_o R $. We have that $t_o>0$ whenever $\alpha,\epsilon\in(0,1)$ while the restriction $\epsilon<\min(\alpha,1-\alpha)$ guarantees that $t_o<1$. The fact that $|R^\text{in}\cap E|/|R^\text{in}|$ satisfies the desired inequalities follows immediately by the definition of $t_o$. Furthermore we can easily estimate
\[
 t_o ^{-n} = \frac{1}{2} \min \Big(\frac{1-\epsilon+\alpha}{\alpha},\frac{2-\epsilon-\alpha}{1-\alpha}  \Big)\sim \frac{1}{\max(\alpha,1-\alpha)}
\]
and thus $t_o\sim_n \max(\alpha,1-\alpha)^\frac{1}{n}$. This gives $|R^\text{in} |\sim_n |R|$ as we wanted.
\end{proof}
%%%%%%%%%%%%%%%%%%%%%%%%%%%%%% PROOF PROOF PROOF
The following theorem constitutes the heart of the matter in this paper regarding the embedding of halo sets associated with convex bases.

%%%%%%%%%%%%%%%%%%%%%%%%%%%%%% THEOREM THEOREM THEOREM
\begin{thm}\label{t.t4}  Let $\mathcal B$ be a homothecy invariant basis consisting of convex sets. For every $\alpha \in(0,1)$ and every measurable set $E\subset \mathbb R^n$ of finite measure we have
\[
\mathcal H_{\mathcal B,\alpha}(E)\subseteq \mathcal{H}_{\mathcal{B}, \alpha(1 +c_n\min(\alpha,1-\alpha)^{2n} \delta)}(\mathcal{H}_{\mathcal{B}, 1 - 3n^{3n/2} \delta }(E)),
\]
for all $\delta\lesssim_n 1-\alpha$. Here $c_n>0$ is a numerical constant that depends only upon the dimension.
\end{thm}
%%%%%%%%%%%%%%%%%%%%%%%%%%%%%% THEOREM THEOREM THEOREM

%%%%%%%%%%%%%%%%%%%%%%%%%%%%%% FIGURE FIGURE FIGURE
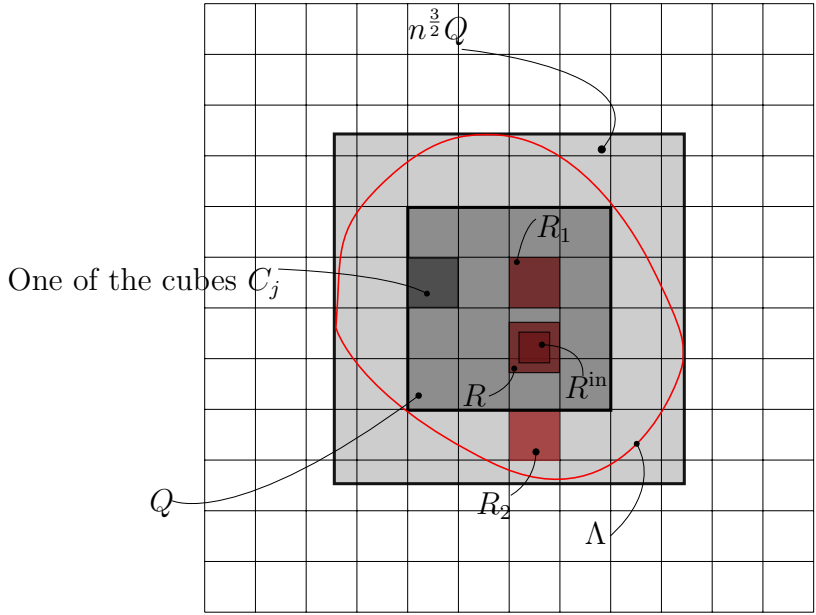
\begin{figure}[htb]
\centering
 \def\svgwidth{0.93\textwidth}
 \hspace*{0.85cm}
\begingroup%
  \makeatletter%
  \providecommand\color[2][]{%
    \errmessage{(Inkscape) Color is used for the text in Inkscape, but the package 'color.sty' is not loaded}%
    \renewcommand\color[2][]{}%
  }%
  \providecommand\transparent[1]{%
    \errmessage{(Inkscape) Transparency is used (non-zero) for the text in Inkscape, but the package 'transparent.sty' is not loaded}%
    \renewcommand\transparent[1]{}%
  }%
  \providecommand\rotatebox[2]{#2}%
  \ifx\svgwidth\undefined%
    \setlength{\unitlength}{214.96117024bp}%
    \ifx\svgscale\undefined%
      \relax%
    \else%
      \setlength{\unitlength}{\unitlength * \real{\svgscale}}%
    \fi%
  \else%
    \setlength{\unitlength}{\svgwidth}%
  \fi%
  \global\let\svgwidth\undefined%
  \global\let\svgscale\undefined%
  \makeatother%
  \begin{picture}(1,0.58530249)%
    \put(0,0){\includegraphics[width=\unitlength]{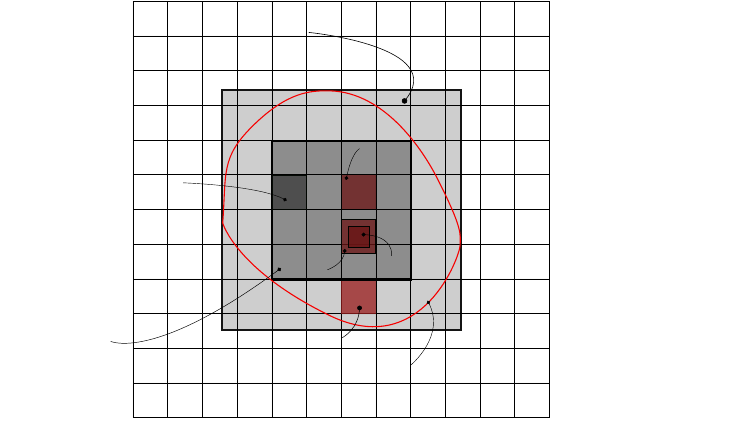}}%
    \put(-0.00250772,0.32139434){\color[rgb]{0,0,0}\makebox(0,0)[lb]{\smash{One of the cubes $C_j$}}}%
    \put(0.36450806,0.55137194){\color[rgb]{0,0,0}\makebox(0,0)[lb]{\smash{$n^{\frac{3}{2}}Q$}}}%
    \put(0.12758913,0.11840823){\color[rgb]{0,0,0}\makebox(0,0)[lb]{\smash{$Q$}}}%
    \put(0.52646988,0.08802234){\color[rgb]{0,0,0}\makebox(0,0)[lb]{\smash{$\Lambda$}}}%
    \put(0.50792865,0.22239208){\color[rgb]{0,0,0}\makebox(0,0)[lb]{\smash{$R ^\text{in}$}}}%
    \put(0.48194387,0.36939537){\color[rgb]{0,0,0}\makebox(0,0)[lb]{\smash{$R_1$}}}%
    \put(0.42698368,0.11552889){\color[rgb]{0,0,0}\makebox(0,0)[lb]{\smash{$R_2$}}}%
    \put(0.41422395,0.21441735){\color[rgb]{0,0,0}\makebox(0,0)[lb]{\smash{$R$}}}%
  \end{picture}%
\endgroup%

\caption[Schematic for the proofs of \S\ref{s.convex}]{Schematic for the proofs of \S\ref{s.convex}}\label{f.schematic}
\end{figure}
%%%%%%%%%%%%%%%%%%%%%%%%%%%%%% FIGURE FIGURE FIGURE

%%%%%%%%%%%%%%%%%%%%%%%%%%%%%% PROOF PROOF PROOF
\begin{proof} Suppose that $E \subset \mathbb{R}^n$ is of finite measure and $\alpha\in(0,1)$.  Let $x\in\mathcal H_{\mathcal B,\alpha}(E)$ and consider a convex set $\tilde\Lambda\in \mathcal B$ such that $x\in \tilde \Lambda$ and
\[
\frac{1}{|\tilde \Lambda|}\int_{\tilde \Lambda}\chi_E>\alpha.
\]
By considering a homothecy $\Lambda$ of $\tilde{\Lambda}$ satisfying $\Lambda\supseteq \tilde \Lambda\ni x$ we have
\begin{equation}\label{e.avg}
\frac{1}{|\Lambda|}\int_{\Lambda}\chi_E=\alpha.
\end{equation}
Using the F. John lemma we can find a rectangular parallelepiped $Q_\Lambda\subset \mathbb R ^n$ such that $Q_{\Lambda}\subseteq \Lambda \subseteq n^\frac{3}{2}
Q_{\Lambda}$. Finally, by invariance under affine transformations, we can map $Q_{\Lambda}$ onto the unit cube $Q=[0,1]^n$ through a bijective linear  transformation
reducing the problem to the case that $\Lambda$ satisfies \eqref{e.avg} and $Q\subseteq \Lambda \subseteq n^\frac{3}{2} Q$, as in Lemma~\ref{l.l3}. It thus suffices
to show that
\[
\Lambda \subseteq \mathcal{H}_{\mathcal{B}_\Lambda, \alpha(1 +c_n\min(\alpha,1-\alpha)^{2n} \delta)}(\mathcal{H}_{\mathcal{B}_{\Lambda}, 1 - 3 n^{3n/2} \delta }(E)),
\]
where $\mathcal B_\Lambda$ is the basis consisting of all the homothecies of $\Lambda$.

With these notations and reductions taken as understood we now set $\epsilon\coloneqq \frac{1}{2}\min(\alpha,1-\alpha)$	 and apply Lemma~\ref{l.l3} to get a cube $R\subset \Lambda$ with $|R|\sim_n \epsilon ^n$ and such that
\begin{equation}\label{e.Ravg}
\frac{\alpha-\epsilon}{1-\epsilon}\leq \frac{1}{|R|}\int_R \chi_E \leq \frac{\alpha}{1-\alpha}.
\end{equation}

For technical reasons we will have to consider the smaller cube $R^\text{in}=t_o R$ provided by Lemma~\ref{l.rin}, where we remember that $t_o\sim_n\max(\alpha,1-\alpha)^\frac{1}{n}$. A schematic associated to the relationships among $R^\text{in}$, $R$, $Q$, and $\Lambda$ is indicated in Figure~\ref{f.schematic}.  Proceeding as in the proof of Theorem~\ref{t.balls}, for $0<\delta \leq \frac{1}{2}\min(1- \frac{\alpha}{1-\epsilon},\frac{1}{3}n^{-\frac{3n}{2}}(1-\alpha))$ we perform a Calder\'on-Zygmund decomposition of $\chi_{E\cap R^\text{in}}$ with respect to the dyadic grid generated by $R^\text{in}$, at level $1-\delta$. This results in the collection of Calder\'on-Zygmund cubes $\{R_j\}_j$ for which $E\cap R^\text{in}\subseteq \cup_j R_j$ a.e.,
\[
\frac{1}{|R_j|}\int_{R_j} \chi_E >1-\delta,
\]
and the cubes $R_j$ are maximal with respect to this property in the dyadic grid generated by $R^\text{in}$. Note that
\[
\frac{1}{|R^\text{in}|}\int_{R^\text{in}} \chi_E \leq \frac{1}{2}+\frac{1}{2}\frac{\alpha}{1-\epsilon}\leq 1-\delta
\]
so the cube $R^\text{in}$ is not itself a Calder\'on-Zygmund cube and thus, for all $j$, $R_j\subsetneq R^\text{in}$ and $R_j ^{(1)}\subseteq R^\text{in}$. We will use these properties in what follows without particular mention.

The rest of the proof is divided into two complementary cases:

\noindent{Case I: for some $j_o$ we have $|4n^{3/2}R_{j_o} \setminus R| > 0$.}

In this case, letting $\operatorname{side}(S)$ denote the sidelength of a cube $S$, we have that
\[
4n^{3/2}\operatorname{side}(R_{j_o}) > \frac{\frac{1}{t_o}\operatorname{side}(R^\text{in}) - \operatorname{side}(R^\text{in})}{2},
\]
implying
\[
|R_{j_o}| \gtrsim_n |R^\text{in}|(t_o^{-1} - 1)^n.
\]
Assuming that $R_{j_o}\eqqcolon T_{j_o} (Q)$ for some homothecy $T_{j_o}$ we set $\Lambda_{j_o} ^\text{small}\coloneqq n^{-\frac{3}{2}}T_{j_o}(\Lambda)$, the dilation taken here with respect to the center of $R_{j_o}$. Then
\[
n^{-\frac{3}{2}}R_{j_o} \subseteq \Lambda_{j_o} ^\text{small} \subseteq R_{j_o}\quad\text{and}\quad |R_{j_o}|\leq n^\frac{3n}{2}|\Lambda_{j_o} ^\text{small}|.
\]
Remembering that $|R_{j_o}\cap E|/|R_{j_o}|>1-\delta$ we can thus calculate
\[
\frac{1}{|\Lambda_{j_o} ^\text{small}|}\int_{\Lambda_{j_o} ^\text{small}}\chi_E  \geq \frac{|\Lambda_{j_o} ^\text{small}| - |\Lambda_{j_o} ^\text{small}\setminus E|}{|\Lambda_{j_o} ^\text{small}|}\geq 1-\frac{|R_{j_o}\setminus E|}{|\Lambda_{j_o} ^\text{small}|}\geq 1- n^\frac{3n}{2}\delta.
\]
On the other hand, since we have that $|\Lambda\cap E|/|\Lambda |=\alpha \leq 1- 6n^\frac{3n}{2}\delta$, we can conclude that there exists a homothecy $\Lambda_{j_o}$ of $\Lambda$ such that
\[
\Lambda_{j_o} ^\text{small} \subseteq \Lambda_{j_o}\subseteq \Lambda\quad \text{and}\quad \frac{1}{|\Lambda_{j_o}|}\int_{\Lambda_{j_o}} \chi_E= 1 - 2 n^{\frac{3n}{2}}\delta .
\]
The measure of $\Lambda_{j_o}$ can be estimated from below as follows:
\[
|\Lambda_{j_o}|\geq |\Lambda_{j_o} ^\text{small}|\geq n^{-\frac{3n}{2}}|R_{j_o}| \gtrsim_n (t_o ^{-1}-1)^n |R^\text{in}|= t_o ^n(t_o ^{-1}-1)^n |R|.
\]
Note that $\Lambda_{j_o} \subset \mathcal{H}_{\mathcal{B}, 1 - 3  n^{3n/2}\delta}(E)$. We conclude that
\[
\begin{split}
|(\mathcal{H}_{\mathcal B_{\Lambda}, 1 - 3  n^{3n/2}\delta }(E) \setminus E) \cap \Lambda |& \geq |(\mathcal{H}_{\mathcal B_{\Lambda}, 1 - 3  n^{3n/2}\delta }(E) \setminus E) \cap \Lambda_{j_o} |
\\
&=|\Lambda_{j_o}\setminus E|\gtrsim_n \delta|\Lambda_{j_o}|\gtrsim_n t_o ^n(t_o ^{-1}-1)^n\epsilon^n \delta |\Lambda|.
\end{split}
\]
Remembering the definitions of $\epsilon$ and $t_o$ and using Lemma~\ref{l.rin} it is not hard to obtain the estimate $t_o ^n(t_o ^{-1}-1)^n \sim_n \min(\alpha,1-\alpha)^n$ for $\alpha\in(0,1)$.

In this case we have thus proved
\[
\Lambda \subseteq \mathcal{H}_{\mathcal{B}_\Lambda, \alpha + c_n\min(\alpha,1-\alpha)^{2n}\delta}(\mathcal{H}_{\mathcal{B}_\Lambda, 1 - 3  n^{3n/2}\delta}(E)).
\]

\noindent{{Case II: for every $j$ we have $4n^{3/2}R_j \subset R$.}}

In this case we apply the Vitali covering lemma to the collection $\{4n^{3/2}R_j\}_j$ resulting in a subcollection $\{R_{j_k}\}_k \subset \{R_j\}_j$ such that the cubes in $\{4n^{3/2} R_{j_k}\}_k$ are a.e. pairwise disjoint and $|\cup_k 4n^{3/2}R_{j_k}| \gtrsim_n |\cup_j 4n^{3/2}R_j|$. Thus
\[
\Big|\bigcup_k R_{j_k}\Big| \sim_n  \Big|\bigcup_k 4n^{3/2} R_{j_k}\Big| \gtrsim_n \Big |\bigcup_j 4n^{3/2}R_j\Big|\geq \Big|\bigcup_j R_j \Big|\geq |E\cap R^\text{in}|.
\]
By Lemma~\ref{l.rin} the measure $|E\cap R^\text{in}|$ can be estimated from below by
\[
|E\cap R^\text{in}| \geq\frac{1}{2}\frac{\alpha-\epsilon}{1-\epsilon} |R^\text{in}|\gtrsim_n \frac{\alpha-\epsilon}{1-\epsilon}|R|\gtrsim_n \frac{\alpha-\epsilon}{1-\epsilon}\epsilon^n |\Lambda|.
\]
We thus have showed that
\begin{equation}\label{e.interrin}
\Big|\bigcup_k R_{j_k}\Big|\gtrsim_n \frac{\alpha-\epsilon}{1-\epsilon}\epsilon^n |\Lambda|.
\end{equation}

Now if $R_{j_k} \eqqcolon T_k (Q)$ for some homothecy $T_k$ we define $\Lambda_k ^{\text{small}}\coloneqq n^{-\frac{3}{2}}T_k(\Lambda)$; the dilation is considered here with respect to the center of the cube $R_{j_k}$. Furthermore, if $R_{j_k} ^{(1)}\eqqcolon T_k ^{(1)} (Q)$ for some homothecy $T_k ^{(1)}$ we define $\Lambda_k ^{\text{big}}\coloneqq T_k ^{(1)}(\Lambda)$. These definitions imply
\[
\begin{split}
 n^{-\frac{3}{2}}R_{j_k} \subseteq \Lambda_k ^{\text{small}} \subseteq R_{j_k} \quad &\text{and}\quad |\Lambda_k ^\text{small}|\geq n^{-\frac{3n}{2}}|R_{j_k}|,
\\
 R_{j_k} ^{(1)} \subseteq \Lambda_k ^{\text{big}} \subseteq n^{\frac{3}{2}} R_{j_k} ^{(1)}\subseteq 4n^\frac{3}{2}R_{j_k}\subseteq R \quad &\text{and}\quad |R_{j_k} ^{(1)}|\geq n^{-\frac{3n}{2}}|\Lambda_k ^\text{big}|.
\end{split}
\]
Using the fact that $|R_j\cap E|/|R_j|>1-\delta$ and the estimates resulting from the definitions above we can estimate
\[
\begin{split}
\frac{1}{|\Lambda_k ^\text{small}|}\int_{\Lambda_k ^\text{small}}\chi_E  &=\frac{|\Lambda_k ^\text{small}| - |\Lambda_k ^\text{small}\setminus E|}{|\Lambda_k ^\text{small}|} >	\frac{|\Lambda_k ^\text{small}| -|R_{j_k}\setminus E| }{|\Lambda_k ^\text{small}|}
\\
&\geq  \frac{|\Lambda_k ^\text{small}|- \delta|R_{j_k}| }{|\Lambda_k ^\text{small}| }\geq 1 -  n^{3n/2}\delta.
\end{split}
\]
For the parent cubes $R_{j_k} ^{(1)}$ we have $|R_{j_k} ^{(1)}\cap E|/|R_{j_k} ^{(1)}|\leq 1-\delta$ by the Calder\'on-Zygmund decomposition. Thus for $\Lambda_k ^\text{big}$ we can estimate
\[
\begin{split}
\frac{1}{|\Lambda_k ^\text{big}|}\int_{\Lambda_k ^\text{big} } \chi_E &= \frac{|\Lambda_k ^\text{big}| - |\Lambda_k ^\text{big}\setminus E|}{|\Lambda_k ^\text{big}|} \leq \frac{|\Lambda_k ^\text{big}|-|R_{j_k} ^{(1)}\setminus E|}{|\Lambda_k ^\text{big}|}
\\
&\leq \frac{|\Lambda_k ^\text{big}|- \delta |R_{j_k} ^{(1)}|}{|\Lambda_k ^\text{big}|} \leq 1 -  n^{-3n/2}\delta.
\end{split}
\]
So for every $k$ there exists a homothecy $\Lambda_k$ of $\Lambda$ such that
\[
 \Lambda_k ^\text{small}\subseteq \Lambda_k \subseteq \Lambda_k ^\text{big}  \subseteq  R\subseteq \Lambda
\]
and
\[
1 - \delta n^{3n/2}< \frac{1}{|\Lambda_k|}\int_{\Lambda_k} \chi_E \leq 1 -  n^{-3n/2}\delta.
\]
We get
\[
|\left(\mathcal{H}_{\mathcal{B}_\Lambda, 1 - \delta n^{3n/2}}(E) \setminus E\right) \cap \Lambda_k | \gtrsim_n \delta |\Lambda_k|.
\]
Note that for each $k$ we have $n^{-\frac{3}{2}}R_{j_k}\subseteq \Lambda_k\subseteq 4n^\frac{3}{2}R_{j_k}$ so the $\Lambda_k$'s are a.e. pairwise disjoint in $R$ and $|\Lambda_k|\sim_n |R_{j_k}|$. We can therefore estimate
\[
\begin{split}
|(\mathcal{H}_{\mathcal{B}_\Lambda, 1 - \delta n^{3n/2}}(E) \setminus E) \cap \Lambda | &\gtrsim_n \delta \sum_k|\Lambda_k|\sim_n \delta \sum_k |R_{j_k}| \gtrsim_n \delta \epsilon^n \frac{\alpha-\epsilon}{1-\epsilon}|\Lambda|
\end{split}
\]
where in the last estimate we used \eqref{e.interrin}.

Since $\frac{\alpha-\epsilon}{1-\epsilon}\gtrsim  \alpha $ for $\alpha\in(0,1)$ we thus have
\[
\Lambda \subseteq \mathcal{H}_{\mathcal{B}_\Lambda, \alpha +c_n\min(\alpha,1-\alpha)^n\alpha \delta}(\mathcal{H}_{\mathcal{B}_\Lambda, 1 -  n^{3n/2}\delta}(E)).
\]
for a dimensional constant $c_n>0$.

The two complementary cases studied above imply that for all $\delta\lesssim_n 1-\alpha$ we have
\[
\mathcal H_{\mathcal B _\Lambda	,\alpha}(E)\subseteq \mathcal{H}_{\mathcal{B} _\Lambda, \alpha(1 +c_n\min(\alpha,1-\alpha)^{2n}  \delta)}(\mathcal{H}_{\mathcal{B}_\Lambda, 1 - 3 n^{3n/2}\delta}(E)),
\]
where $c_n>0$ is a numerical constant depending only on the dimension.
\end{proof}
%%%%%%%%%%%%%%%%%%%%%%%%%%%%%% PROOF PROOF PROOF
By an argument along the lines of the proof of Lemma \ref{l.l1} we then have the following.
%%%%%%%%%%%%%%%%%%%%%%%%%%%%%% THEOREM THEOREM THEOREM
\begin{thm}
Let $\mathcal{B}$ be a homothecy invariant density basis of convex sets in $\mathbb{R}^n$.  Suppose $C_{\mathcal{B}}(\alpha) \lesssim_{\mathcal B} (1 - \alpha)^p$ holds for $\alpha$ sufficiently close to $1$ and for some fixed $0 < p < 1$. Then
\[
C_{\mathcal{B}}  \in C^p(0,1)
\]
and consequently
\[
\phi_{\mathcal{B}} \in C^p([0,\infty)).
\]
\end{thm}
%%%%%%%%%%%%%%%%%%%%%%%%%%%%%% THEOREM THEOREM THEOREM

%%%%%%%%%%%%%%%%%%%%%%%%%%%%%% REMARK REMARK REMARK
\begin{remark}  It is quite possible that we should be able to have improved smoothness results for $C_{\mathcal{B}}(\alpha)$ for $0 < \alpha < 1$, especially in the cases that the maximal operator $M_{\mathcal{B}}$ is the Hardy-Littlewood or strong maximal operator.   Arguments along these lines would have to be substantially more sophisticated than what we have provided here, however, as the known Solyanik estimates for the strong maximal operator are known to be sharp; see \cite{hp13}.  This is a subject of ongoing research.
\end{remark}
%%%%%%%%%%%%%%%%%%%%%%%%%%%%%% REMARK REMARK REMARK

%%%%%%%%%%%%%%%%%%%%%%%%%%%%%% SECTION  SECTION SECTION
%%%%%%%%%%%%%%%%%%%%%%%%%%%%%% SECTION  SECTION SECTION: Bibliography
 

\begin{bibsection}
 \begin{biblist}

\bib{bh}{article}{
   author={Beznosova, Oleksandra},
   author={Hagelstein, Paul Alton},
   title={Continuity of halo functions associated to homothecy invariant
   density bases},
   journal={Colloq. Math.},
   volume={134},
   date={2014},
   number={2},
   pages={235--243},
   issn={0010-1354},
   review={\MR{3194408}},
}

\bib{busemannfeller1934}{article}{
author = {Busemann, H.},
author = {Feller, W.},
journal = {Fundamenta Mathematicae},
language = {ger},
number = {1},
pages = {226-256},
publisher = {Institute of Mathematics Polish Academy of Sciences},
title = {Zur Differentiation der Lebesgueschen Integrale},
url = {http://eudml.org/doc/212688},
volume = {22},
year = {1934},
}	

\bib{cab}{article}{
   author={Cabrelli, Carlos},
   author={Lacey, Michael T.},
   author={Molter, Ursula},
   author={Pipher, Jill C.},
   title={Variations on the theme of Journ\'e's lemma},
   journal={Houston J. Math.},
   volume={32},
   date={2006},
   number={3},
   pages={833--861},
   issn={0362-1588},
   review={\MR{2247912 (2007e:42011)}},
}
	

\bib{CorF}{article}{
   author={C\'ordoba, A.},
   author={Fefferman, R.},
   title={On the equivalence between the boundedness of certain classes of
   maximal and multiplier operators in Fourier analysis},
   journal={Proc. Nat. Acad. Sci. U.S.A.},
   volume={74},
   date={1977},
   number={2},
   pages={423--425},
   issn={0027-8424},
   review={\MR{0433117 (55 \#6096)}},
}


\bib{laceyferg2002}{article}{
   author={Ferguson, Sarah H.},
   author={Lacey, Michael T.},
   title={A characterization of product BMO by commutators},
   journal={Acta Math.},
   volume={189},
   date={2002},
   number={2},
   pages={143--160},
   issn={0001-5962},
   review={\MR{1961195 (2004e:42026)}},
}

\bib{gilbargtrudinger}{book}{
   author={Gilbarg, David},
   author={Trudinger, Neil S.},
   title={Elliptic partial differential equations of second order},
   series={Grundlehren der Mathematischen Wissenschaften [Fundamental
   Principles of Mathematical Sciences]},
   volume={224},
   edition={2},
   publisher={Springer-Verlag, Berlin},
   date={1983},
   pages={xiii+513},
   isbn={3-540-13025-X},
   review={\MR{737190 (86c:35035)}},
}

\bib{Gu}{article}{
   author={de Guzm{\'a}n, Miguel},
   title={Differentiation of integrals in ${\bf R}^{n}$},
   conference={
      title={Measure theory},
      address={Proc. Conf., Oberwolfach},
      date={1975},
   },
   book={
      publisher={Springer},
      place={Berlin},
   },
   date={1976},
   pages={181--185. Lecture Notes in Math., Vol. 541},
   review={\MR{0476978 (57 \#16523)}},
}



\bib{hlp}{article}{
		Author = {Hagelstein, P. A.},
		Author = {Luque, T.},
		Author = {Parissis, I.},
		Eprint = {1304.1015},
		Title = {Tauberian conditions, Muckenhoupt weights, and differentiation properties of weighted bases},
		Url = {http://arxiv.org/abs/1304.1015},
		journal={Trans. Amer. Math. Soc.},
		Year = {to appear}}

\bib{hp13}{article}{
			Author = {Hagelstein, Paul A.},
			Author = {Parissis, Ioannis},
			%Eprint = {1310.3771},
			Title = {Solyanik estimates in harmonic analysis},
			%Url = {http://arxiv.org/abs/1310.3771},
			journal={Springer Proc. Math. Stat.},
            volume={108},
            date={2014},
            pages={87--103},
            review={\MR{3297656}}
            }

			

\bib{hp14}{article}{
			Author = {Hagelstein, Paul A.},
			Author = {Parissis, Ioannis},
			Eprint = {1405.6631},
			Title = {Weighted Solyanik Estimates for the Hardy-Littlewood maximal operator and embedding of $A_\infty$ into $A_p$},
			Url = {http://arxiv.org/abs/1405.6631},
			journal={J. Geom. Anal.},
			year={to appear},
			}%Year = {2014}}
			
\bib{hs2}{article}{
   author={Hagelstein, Paul Alton},
   author={Stokolos, Alexander},
   title={An extension of the C\'ordoba-Fefferman theorem on the equivalence
   between the boundedness of certain classes of maximal and multiplier
   operators},
   language={English, with English and French summaries},
   journal={C. R. Math. Acad. Sci. Paris},
   volume={346},
   date={2008},
   number={19-20},
   pages={1063--1065},
   issn={1631-073X},
   review={\MR{2462049 (2010a:42064)}},
}
	

\bib{hs}{article}{
   author={Hagelstein, Paul Alton} ,
   author={Stokolos, Alexander} ,
   title={Tauberian conditions for geometric maximal operators},
   journal={Trans. Amer. Math. Soc.},
   volume={361},
   date={2009},
   number={6},
   pages={3031--3040},
   issn={0002-9947},
   review={\MR{2485416 (2010b:42023)}},
}

\bib{hardylittlewood}{article}{
   author={Hardy, G. H.},
   author={Littlewood, J. E.},
   title={A maximal theorem with function-theoretic applications},
   journal={Acta Math.},
   volume={54},
   date={1930},
   number={1},
   pages={81--116},
   issn={0001-5962},
   review={\MR{1555303}},
}

\bib{Hayes}{article}{
   author={Hayes, C. A., Jr.},
   title={A condition of halo type for the differentiation of classes of
   integrals},
   journal={Canad. J. Math.},
   volume={18},
   date={1966},
   pages={1015--1023},
   issn={0008-414X},
   review={\MR{0199318 (33 \#7466)}},
}

\bib{fj}{article}{
   author={John, Fritz},
   title={Extremum problems with inequalities as subsidiary conditions},
   conference={
      title={Studies and Essays Presented to R. Courant on his 60th
      Birthday, January 8, 1948},
   },
   book={
      publisher={Interscience Publishers, Inc., New York, N. Y.},
   },
   date={1948},
   pages={187--204},
   review={\MR{0030135 (10,719b)}},
}

\bib{laceyterwilleger}{article}{
   author={Lacey, Michael },
   author={Terwilleger, Erin},
   title={Hankel operators in several complex variables and product BMO},
   journal={Houston J. Math.},
   volume={35},
   date={2009},
   number={1},
   pages={159--183},
   issn={0362-1588},
   review={\MR{2491875 (2010c:47071)}},
}



\bib{solyanik}{article}{
   author={Solyanik, A. A.},
   title={On halo functions for differentiation bases},
   language={Russian, with Russian summary},
   journal={Mat. Zametki},
   volume={54},
   date={1993},
   number={6},
   pages={82--89, 160},
   issn={0025-567X},
   translation={
      journal={Math. Notes},
      volume={54},
      date={1993},
      number={5-6},
      pages={1241--1245 (1994)},
      issn={0001-4346},
   },
   review={\MR{1268374 (95g:42033)}},
}

\bib{soria}{article}{
   author={Soria, Fernando},
   title={Note on differentiation of integrals and the halo conjecture},
   journal={Studia Math.},
   volume={81},
   date={1985},
   number={1},
   pages={29--36},
   issn={0039-3223},
   review={\MR{818168 (87j:42058)}},
}

 \bib{SI}{book}{
   author={Stein, Elias M.},
   title={Singular integrals and differentiability properties of functions},
   series={Princeton Mathematical Series, No. 30},
   publisher={Princeton University Press},
   place={Princeton, N.J.},
   date={1970},
   pages={xiv+290},
   review={\MR{0290095 (44 \#7280)}},
}

\end{biblist}
\end{bibsection}
\end{document}